\documentclass[a4paper]{amsart}
\usepackage{amsmath}
\usepackage{amsthm}
\usepackage{amssymb}
\usepackage{color}
\usepackage{graphicx}
\usepackage{tikz}
\usepackage{subfigure}
\usepackage{epstopdf}
\usepackage{cases}
\usepackage{bm}
\usepackage{euscript}
\newcommand{\R}{\mathbb{R}}
\newcommand{\D}{\mathbb{D}}
\newcommand{\C}{\mathbb{C}}
\renewcommand{\H}{\mathcal{H}}
\newcommand{\N}{\mathbb{N}}

\theoremstyle{plain}
\newtheorem{thm}{Theorem}[section]
\newtheorem*{thm*}{Theorem}
\newtheorem{lem}[thm]{Lemma}
\newtheorem{prop}[thm]{Proposition}
\newtheorem{cor}[thm]{Corollary}
\theoremstyle{definition}

\newtheorem{rmk}[thm]{Remark}

\numberwithin{equation}{section}

\newcommand{\abs}[1]{\left|#1\right|}
\def\a{\alpha}        \def\b{\beta}      \def\th{\theta}       
\def\vp{\varphi}      \def\z{\zeta}
\def\zth{\z_{\th}}  
\def\ra{\rightarrow}      \def \ov{\overline}
\def\beq{\begin{equation}}  \def\eeq{\end{equation}}
\def\beqq{\begin{equation*}}  \def\eeqq{\end{equation*}}
\def\bprof{\begin{proof}}    \def\eprof{\end{proof}}
\def\bad{\begin{aligned}}    \def\ead{\end{aligned}}
\def\l{\left}      \def\r{\right}
\def\bthm{\begin{thm}} \def\ethm{\end{thm}}
\def\blem{\begin{lem}} \def\elem{\end{lem}}
\def\p{\partial} \def\e{e^{i \th}}
\def\be{\begin{enumerate}} \def\ee{\end{enumerate}}

\begin{document}

\title[variability regions for the third derivative]{Variability regions for the third derivative of bounded analytic functions}

\author{Gangqiang Chen}
\address{Graduate School of Information Sciences,
Tohoku University,
Aoba-ku, Sendai 980-8579, Japan}
\email{cgqmath@ims.is.tohoku.ac.jp; cgqmath@qq.com}

\begin{abstract}
Let $z_0$ and $w_0$ be given points in the open unit disk $\mathbb{D}$
with $|w_0| < |z_0|$, and  $\H_0$ be the class of all analytic self-maps $f$
of $\mathbb{D}$ normalized by $f(0)=0$.
In this paper, we establish the third order Dieudonn\'e Lemma, and apply it to explicitly determine the variability region
$\{f'''(z_0): f\in \H_0,f(z_0) =w_0, f'(z_0)=w_1\}$ for given $z_0,w_0,w_1$ and give the form of all the extremal functions.
\end{abstract}

\maketitle

\section{Introduction}
We denote by $\mathbb{D}=\{z\in \C: |z|<1\}$ the open unit disk
in the complex plane $\C$ and by $\H_0$ the set of all analytic self-maps $f$ of $\mathbb{D}$ normalized by $f(0)=0$. In 1890, Schwarz proved that $|f(z_0)|\le |z_0|$ and $|f'(0)|\le 1$ hold for all $f\in \H_0$ and $z_0\in \D$, which gives sharp estimates of the values of $f(z_0)$ and $f'(0)$.
 Since the discovery of the celebrated Schwarz Lemma, a lot of famous mathematicians have devoted themselves to the extensions and generalizations of Schwarz's Lemma.

It is worth mentioning the refinements of Schwarz's Lemma, before that we fix some notation. For $c\in\C$ and $\rho>0$, we define the discs $\D(c, \rho)$ and $\overline{\D}(c, \rho)$ by
$\D(c, \rho):=\left\{ \zeta \in \C : |\zeta-c|< \rho \right\}$,
and
$\overline{\D}(c, \rho):=\left\{\zeta \in \C : |\zeta-c|\le \rho \right\}$.
Let $z_0,w_0\in \D$ be given points
with $|w_0|<|z_0|$. Then Schwarz's Lemma can be restated as $\{f(z_0):f \in \mathcal{H}_0\}=\ov{\D}(0,\ |z_0|)$ for any  $z_0\in \mathbb{D}\setminus \{0\}$, and $f(z_0)\in \partial \D(0,\ |z_0|)$ if and only if
 $ f$ is a rotation about the origin.
 In 1934, Rogosinski \cite{rogosinski1934} established an assertion which can be considered as a
sharpened version of Schwarz's Lemma. His result describes the variability region of $f(z)$ for $z\in\D$, $f\in \H_0$ with $|f'(0)|<1$, proved by calculating the envelop of a certain union of disks (for the details of the proof, see \cite{duren1983univalent} and \cite{goluzin1969geometric}). In 1996, Mercer \cite{mercer1997sharpened} determined the variability region of $f(z)$ for $z\in\D$, $f\in \H_0$ with $f(z_0)=w_0 (z_0\ne 0)$, which can be reduced to Rogosinski's Lemma as $z_0\ra 0$.

In 1931, Dieudonn\'e \cite{dieudonne1931} first obtained a sharp inequality for the derivative $f'(z_0)$ of $f\in \H_0$,
\begin{equation}\label{ineq:f'}
\abs{f'(z_0)-\frac{w_0}{z_0}}\le \frac{|z_0|^2 -|w_0|^2}{|z_0|(1-|w_0|^2)},
\end{equation}
which is an improvement for the derivative part of Schwarz's Lemma.
Equality in \eqref{ineq:f'} holds if and only if $f$ is a Blaschke product of degree 2 fixing 0. Here we remark that a Blaschke product of degree $n \in \N$  takes the form
         $$ B(z)=e^{i \theta}\prod\limits_{j=1}^{n}
         \frac{z-z_j}{1-\overline{z_j}z}, \quad z, z_j\in \D, \theta \in \mathbb{R}.$$
Moreover, his result, which is nowadays known as Dieudonn\'e's Lemma, coincides with the description of  the variability region of $f'(z_0)$, $f\in \mathcal{H}_0$, at a fixed point $z_0\in \mathbb{D}$. In other words,
if we define the M\"obius transformation
$$T_{a}(z)=\frac{z+a}{1+\overline{a}z},\quad z, a\in\D,$$
 and write
 \[
  \Delta (z_0,w_0)
  =
  \overline{\mathbb{D}}
  \left( \frac{w_0}{z_0} , \frac{|z_0|^2-|w_0|^2}{|z_0|(1-|w_0|^2)} \right),
\]
then his observation can be restated as $\{f'(z_0):f \in \mathcal{H}_0, f(z_0)=w_0\}=\Delta(z_0,w_0)$, and
$f'(z_0) \in \partial \Delta (z_0,w_0) $ if and only if
$f(z)=z\;T_{u_0}(e^{i \theta}T_{-z_0}(z))$, where $u_0=w_0/z_0$ and $\theta \in \R$  (see also \cite{beardon2004multi}, \cite{chen_2019} and \cite{rivard2013application}) .

In 2013, Rivard \cite{rivard2013application} proved the so-called second order Dieudonn\'e Lemma which tells us that
if $f\in\mathcal{H}_0$  is not an automorphism of $\mathbb{D}$, then
   \begin{align}\label{ineq:f''}
       &\left|\frac{1}{2}z_0^2 f''(z_0)-\frac{z_0 w_1-w_0}{1-|z_0|^2}
         +\frac{\overline{w_0}(z_0 w_1-w_0)^2}{|z_0|^2-|w_0|^2}\right| +\frac{|z_0||z_0 w_1-w_0|^2}{|z_0|^2-|w_0|^2}\nonumber\\
       &\qquad \le \frac{|z_0| (|z_0|^2-|w_0|^2)}{(1-|z_0|^2)^2},
   \end{align}
where $f(z_0)=w_0$ and $f'(z_0)=w_1\in \Delta(z_0,w_0)$. Equality in \eqref{ineq:f''} holds if and only if $f(z)=zg(z)$ where $g(z)$ is a Blaschke product of degree 1 or 2 (see also \cite{cho2012multi}).
The original version can be appropriately modified as follows. Let $|z_0|=r$, $|w_0|=s$ and $\beta \in \ov{\D}$. Then
\begin{align*}
V(z_0,w_0,\beta)&=\{f''(z_0):f\in\H_0,f(z_0)=w_0,
f'(z_0)=\frac{w_0}{z_0}+\frac{r^2-s^2}{z_0(1-r^2)}\b\}\\
&=\frac{2(r^2-s^2)}{r^2(1-r^2)^2}\ov{\D}(c(\b),\rho(\b)),
\end{align*}
where
$$
c(\b)=\frac{\ov{z}_0}{z_0}\beta(1-\overline{w_0}\beta), \quad \rho(\b)=r(1-|\beta|^2),
$$
and for $\beta\in \D$, $f\in \p V(z_0,w_0,\beta)$ if and only if $f(z)=z T_{u_0}\l(T_{-z_0}(z) T_{v_0}(e^{i \theta}T_{-z_0}(z))\r)$, where $\theta \in \R$, $u_0=w_0/z_0$ and $v_0=\overline{z}_0^2 \beta/r^2$.
By using this result, the author \cite{chen_2019} obtained the sharp upper bound for $|f''(z_0)|$ depending only on $|z_0|$. In addition, the author and Yanagihara \cite{chen2020} also maked use of this consequence to precisely determine the variability region
$V(z_0,w_0)=\{f''(z_0):f\in\H_0, f(z_0)=w_0\}$.

 It is natural for us to further study the third order derivative $f'''$ of $f\in \H_0$. In fact, the purpose of this present paper is to establish a third order Dieudonn\'e Lemma and then apply it to a variability region problem. Before the statement of our main result,
we denote $c$ and $\rho$ by
\begin{equation*}
\left\{
\begin{aligned}
c&=c(z_0,w_0,w_1,w_2)
=\frac{6(r^2-s^2)}{z_0^3(1-r^2)^3}\left(\mathcal{B}+z_0\mu(1-|\lambda|^2)(1+r^2-2 \ov{w_0}\lambda-z_0 \ov{\lambda}\mu)\right);\\
\rho&=\rho(z_0,w_0,w_1,w_2)=\frac{6(r^2-s^2)}{r(1-r^2)^3}(1-|\lambda|^2)(1-|\mu|^2),
\end{aligned}
\right.
\end{equation*}
where
$$\mathcal{B}=\ov{w_0}^2 \lambda^3-\ov{w_0}(1+r^2)\lambda^2+r^2\lambda.$$
\begin{thm}[The third order Dieudonn\'e Lemma]\label{thm:lemthird}
Let $z_0, w_0\in \mathbb{D}$, $\lambda, \mu \in \ov{\D}$ with $|w_0|=s<r=|z_0|$, $w_1= \dfrac{w_0}{z_0}+\dfrac{r^2-s^2}{z_0(1-r^2)}\lambda$, $$w_2=\dfrac{2(r^2-s^2)\lambda(1-\overline{w_0}\lambda)}{z_0^2(1-r^2)^2}+\dfrac{2(r^2-s^2)(1-|\lambda|^2)}{z_0(1-r^2)^2}\mu.$$ Suppose that $f\in\mathcal{H}_0$, $f(z_0) = w_0$, $f'(z_0)=w_1$ and $f''(z_0)=w_2$.
Set $u_0=w_0/z_0$, $v_0=r^2\lambda/z_0^2$.
\begin{enumerate}
\item If $|\lambda|=1$, then $f'''(z_0)=c$ and $f(z)=z T_{u_0}(v_0 T_{-z_0}(z))$.

\item If $|\lambda|<1$, $|\mu|=1$, then $f'''(z_0)=c$ and $f(z)=z T_{u_0}\l(T_{-z_0}(z) T_{v_0}(\tau T_{-z_0}(z))\r)$, where $\tau=\ov{z_0}\mu/z_0$.

\item If $|\lambda|<1$, $|\mu|<1$, then the region of values of $f'''(z_0)$ is the closed disk
$\overline{\D}(c, \rho)$.
Furthermore, $f'''(z_0)\in \p\D(c, \rho)$ if and only if\\
$f(z)=z T_{u_0}\l(T_{-z_0}(z) T_{v_0}(T_{-z_0}(z) T_{\eta}(\e T_{-z_0}(z)))\r)$, where $\theta \in \R$ and
$$\eta=\frac{r^2\mu}{z_0^2}+\frac{\lambda^2r^2(r^2w_0-z_0^2\ov{w_0})}{z_0^5(1-|\lambda|^2)}.$$
\end{enumerate}
\end{thm}
In Sect. 4, we will make use of the third order Dieudonn\'e Lemma to determine the region of values of $f'''(z_0)$, $f\in\H_0$, in terms of $z_0, f(z_0),f'(z_0)$. More precisely, we shall explicitly describe the variability region $\{f'''(z_0):f\in\H_0,f(z_0)=w_0,f'(z_0)=w_1\}$ for given points $z_0,w_0,w_1$,  and give the form of all the extremal functions. For this purpose, we restate Case (3) in Theorem \ref{thm:lemthird} as follows. Under the same hypotheses as in Theorem \ref{thm:lemthird} except that $\lambda\in \D$, then
$$
V(z_0,w_0,\lambda,\mu)=\{f'''(z_0):f\in \H_0, f(z_0)=w_0, f'(z_0)=w_1, f''(z_0)=w_2\}=\overline{\D}(c, \rho).
$$

The study on the third derivative of bounded analytic functions in this paper is not exhaustive but could, in our opinion, serve as a basis for further investigations such as the subordination and the extremal problems.
\section{Proof of the third order Dieudonn\'e Lemma}
We begin this section with some fundamental knowledge which is convenient for
understanding the proof of Theorem \ref{thm:lemthird}.
First, we give an introduction to the definition of Peschl invariant derivatives.
For $g:\mathbb{D}\to \mathbb{D}$ holomorphic, Peschl \cite{peschl1955invariants} defined the so-called Peschl's invariant derivatives $D_n g(z)$ with respect to the hyperbolic metric by the Taylor series expansion:
$$z\ra \frac{g(\frac{z+z_0}{1+\overline{z}_0 z})-g(z_0)}{1-\overline{g(z_0)}g(\frac{z+z_0}{1+\overline{z}_0 z})}=\sum_{n=1}^{\infty}\frac{D_n g(z_0)}{n!}z^n,\quad z, z_0\in \D.$$

For example, precise forms of $D_n g(z)$, $n=1,2,3$, are given by
\begin{align*}
D_1 g(z)&=\frac{(1-\abs{z}^2)g'(z)}{1-\abs{g(z)}^2},\\
D_2 g(z)&=\frac{(1-\abs{z}^2)^2}{1-\abs{g(z)}^2}
\Bigg[g''(z)-\frac{2\overline{z}g'(z)}{1-\abs{z}^2}
+\frac{2\overline{g(z)}g'(z)^2} {1-\abs{g(z)}^2}\Bigg],\\
D_3 g(z)&=\frac{(1-\abs{z}^2)^3}{1-\abs{g(z)}^2}
\Bigg[g'''(z)-\frac{6\overline{z}g''(z)}{1-\abs{z}^2}
+\frac{6\overline{g(z)}g'(z)g''(z)} {1-\abs{g(z)}^2}
+\frac{6\overline{z}^2g'(z)}{(1-\abs{z}^2)^2}\\
&\qquad -\frac{12\overline{g(z)}g'(z)^2} {(1-\abs{z}^2)(1-\abs{g(z)}^2)}
+\frac{6\overline{g(z)}^2g'(z)^3} {(1-\abs{g(z)}^2)^2}\Bigg].
\end{align*}

In 2012, Cho, Kim and Sugawa \cite{cho2012multi} proved the following inequality in terms of Peschl's
invariant derivatives, from which we can derive a concrete inequality for $g'''(z)$ in terms of $z, g(z), g'(z)$ and $ g''(z)$.
\begin{lem}[\cite{cho2012multi}]\label{lem:cho}
If $g:\mathbb{D}\to \mathbb{D}$ is holomorphic, then
\begin{equation}
\left|\frac{D_3 g(z)}{6}(1-|D_1 g(z)^2|+\overline{D_1 g(z)}\left(\frac{D_2 g(z)}{2}\right)^2\right|+\left|\frac{D_2 g(z)}{2}\right|^2\leq(1-|D_1 g(z)|^2)^2,
\end{equation}
equality holds for a point $z\in \mathbb{D}$ if and only if $g$ is a Blaschke product of degree at most 3.
\end{lem}

Before giving the proof of Theorem \ref{thm:lemthird}, we have the following argument which helps us to simplify the situation. For brevity, we assume that $z_0=re^{i \vp}, w_0=s e^{i \xi}\in \D$. Define
the `rotation function'
$\tilde{f}(z)=e^{-i \xi}f(e^{i \vp} z)$, then we have $\tilde{f}'(r)=e^{i(\vp-\xi)}f'(z_0)\in \Delta(r,s)$, $\tilde{f}''(r)=e^{i(2\vp-\xi)}f''(z_0)$ and $\tilde{f}'''(r)=e^{i(3\vp-\xi)}f'''(z_0)$.
So we can relabel $\tilde{f}$ as $f$, and assume that
$$z_0=r,w_0=s,w_1=\frac{s}{r}+\frac{r^2-s^2}{r(1-r^2)}\lambda, \quad \lambda\in \ov{\D},$$
\begin{align*}
w_2&=\frac{2(r^2-s^2)}{r^2(1-r^2)^2}(\lambda(1-s\lambda)+r(1-|\lambda|^2)\mu),\quad \mu\in \ov{\D}.
\end{align*}
Correspondingly,
we define $c_0$ and $\rho_0$ by
\begin{equation*}
\left\{
\begin{aligned}
c_0&=c_0(r,s,\lambda,\mu)
=A\left(B+r\mu(1-|\lambda|^2)(1+r^2-2 s\lambda-r \ov{\lambda}\mu)\right);\\
\rho_0&=\rho_0(r,s,\lambda,\mu)=Ar^2(1-|\lambda|^2)(1-|\mu|^2),
\end{aligned}
\right.
\end{equation*}
where
\begin{equation}\label{eq:AB}
A=\frac{6(r^2-s^2)}{r^3(1-r^2)^3},\quad
B=s^2 \lambda^3-s(1+r^2)\lambda^2+r^2\lambda.
\end{equation}
Assume that $g(z)=f(z)/z$, then $g$ is an analytic self-map of $\D$. A straight computations shows that
$D_1 g(r)=\lambda$, $D_2 g(r)=2\mu(1-|\lambda|^2)$ and
$$
D_3 g(r)=\frac{r(1-r^2)^3}{r^3(r^2-s^2)}f'''(r)
+\frac{6b}{r^2},
$$
where $$
b=-s^2\lambda^3+s(1+r^2)\lambda^2-r^2\lambda
     +r\mu(-1-r^2+2s\lambda)(1-|\lambda|^2).
$$
From Lemma \ref{lem:cho}, we have
$$|\frac{D_3 g(r)}{6}+\overline{\lambda}\mu^2(1-|\lambda|^2)|\leq(1-|\lambda|^2)(1-|\mu|^2),$$
Then we obtain
$$|f'''(r)+\frac{6(r^2-s^2)}{r^3(1-r^2)^3}(b+r^2\overline{\lambda}\mu^2(1-|\lambda|^2))|
\leq \frac{6(r^2-s^2)}{r(1-r^2)^3}(1-|\lambda|^2)(1-|\mu|^2),$$
which is
\begin{equation}\label{eq:f'''(r)}
|f'''(r)-c_0|\le \rho_0.
\end{equation}
Equality in \eqref{eq:f'''(r)} holds if and only if $f(z)=zg(z)$, where
$g$ is a Blaschke product of degree $1,2$ or $3$ and satisfies

\begin{equation}\label{condition2}
\left\{
\begin{aligned}
g(r)&=\frac{s}{r};\\
g'(r)&=\frac{r^2-s^2}{r^2(1-r^2)}\lambda;\\
g''(r)&=\frac{2(r^2-s^2)}{r^3(1-r^2)^2}(-s\lambda^2+r^2\lambda+r\mu(1-|\lambda|^2)).
\end{aligned}
\right.
\end{equation}
Now we are ready to give the proof of Theorem \ref{thm:lemthird}.
\begin{proof}[Proof of Theorem \ref{thm:lemthird}]

By using the same method in the proof of \cite[Lemma 2.2]{chen_2019}, we can easily prove Case (1) and (2), so we omit the proofs here.

(3)The inequality \eqref{eq:f'''(r)} means that $f'''(r)$ lies in $\overline{\D}(c_0,\rho_0)$.
 To show that $\overline{\D}(c_0,\rho_0)$ is covered, let $\alpha \in \ov{\D}$, $u_0=s/r$ and set
 $f(z)=zg(z)$, where
 $$g(z)=T_{u_0}\l(T_{-r}(z) T_{\lambda}(T_{-r}(z) T_{\mu}(\a T_{-r}(z)))\r).$$
Then $f(0)=0$ and $f(r)=s$. Next we need to show that $f'(r)=w_1$.
Note that
\begin{equation}\label{eq:T-g}
T_{-u_0}\circ g(z)=T_{-r}(z) T_{\lambda}(T_{-r}(z) T_{\mu}(\a T_{-r}(z))).
\end{equation}
Differentiating both sides, we get
\begin{equation}\label{g'}
\begin{aligned}
(T_{-u_0})'(g(z))g'(z)&=T_{-r}'(z) T_{\lambda}(T_{-r}(z) T_{\mu}(\a T_{-r}(z)))\\
&\quad +T_{-r}(z) T_{\lambda}'(T_{-r}(z) T_{\mu}(\a T_{-r}(z))) \\
&(T_{-r}'(z) T_{\mu}(\a T_{-r}(z))+T_{-r}(z) T_{\mu}'(\a T_{-r}(z))\a T_{-r}'(z))).
\end{aligned}
\end{equation}
for all $z\in \D$.
Substituting $z=r$ into this equation, we have
$$(T_{-u_0})'(g(r))g'(r)=T_{-r}'(z_0) T_{\mu}(0),$$
which implies
$$g'(r)=\frac{(r^2-s^2)\lambda}{r^2(1-r^2)}.$$
Thus, we obtain that $f$ satisfies
$$f'(r)=g(r)+rg'(r)=w_1.$$
Similarly, differentiating both sides of \eqref{g'}, we obtain
\begin{equation}\label{appendix:g''}
\begin{aligned}
&(T_{-u_0})''(g(z)) (g'(z))^2+(T_{-u_0})'(g(z)) g''(z)\\
&\quad=T_{-r}''(z) T_{\lambda}(T_{-r}(z) T_{\mu}(\a T_{-r}(z)))\\
&\quad\quad+2T_{-r}'(z) T_{\lambda}'(T_{-r}(z) T_{\mu}(\a T_{-r}(z)))(T_{-r}'(z) T_{\mu}(\a T_{-r}(z))+T_{-r}(z) T_{\mu}'(\a T_{-r}(z))\a T_{-r}'(z)))\\
&\quad\quad+T_{-r}(z) T_{\lambda}''(T_{-r}(z) T_{\mu}(\a T_{-r}(z)))(T_{-r}'(z) T_{\mu}(\a T_{-r}(z))+T_{-r}(z) T_{\mu}'(\a T_{-r}(z))\a T_{-r}'(z))^2\\
&\quad\quad+T_{-r}(z) T_{\lambda}'(T_{-r}(z) T_{\mu}(\a T_{-r}(z)))\cdot\\
&\qquad\qquad\big(T_{-r}''(z) T_{\mu}(\a T_{-r}(z)))+2T_{-r}'(z) T_{\mu}'(\a T_{-r}(z))\a T_{-r}'(z))\\
&\qquad\qquad +T_{-r}(z) T_{\mu}''(\a T_{-r}(z))(\a T_{-r}'(z))^2+T_{-r}(z) T_{\mu}'(\a T_{-r}(z))\a T_{-r}''(z)\big),\quad z\in \D.
\end{aligned}
\end{equation}
Substituting $z=r$ into the above equation,
\begin{align*}
&(T_{-u_0})''(g(r)) (g'(r))^2+(T_{-u_0})'(g(r))g''(r)\\
&=T_{-r}''(r) T_{\mu}(0)+2T_{-r}'(r) T_{\lambda}'(0)(T_{-r}'(z) T_{\lambda}(\a T_{-r}(z)).
\end{align*}
We get that
$$g''(r)=\frac{2(r^2-s^2)}{r^3(1-r^2)^2}(-s\lambda^2+r^2\lambda+r\mu(1-|\lambda|^2)).$$
The above, in conjunction with $
f''(z)=2g'(z)+z g''(z)$, immediately yields $f''(r)=w_2$.

Next we determine the form of $f'''(r)$. Differentiating both sides of \eqref{appendix:g''},
\begin{equation}\label{appendix:g'''}
\begin{aligned}
&(T_{-u_0})'''(g(z)) (g'(z))^3+3(T_{-u_0})''(g(z)) g''(z)+T_{-u_0}'''(g(z)) g'''(z)\\
&\quad=T_{-r}'''(z) T_{\lambda}(T_{-r}(z) T_{\mu}(\a T_{-r}(z)))\\
&\quad\quad+3T_{-r}''(z) T_{\lambda}'(T_{-r}(z) T_{\mu}(\a T_{-r}(z)))(T_{-r}'(z) T_{\mu}(\a T_{-r}(z))+T_{-r}(z) T_{\mu}'(\a T_{-r}(z))\a T_{-r}'(z)))\\
&\quad\quad+3T_{-r}'(z) T_{\lambda}''(T_{-r}(z) T_{\mu}(\a T_{-r}(z)))(T_{-r}'(z) T_{\mu}(\a T_{-r}(z))+T_{-r}(z) T_{\mu}'(\a T_{-r}(z))\a T_{-r}'(z))^2\\
&\quad\quad+3T_{-r}'(z) T_{\lambda}'(T_{-r}(z) T_{\mu}(\a T_{-r}(z)))\cdot\\
&\qquad\qquad\big(T_{-r}''(z) T_{\mu}(\a T_{-r}(z)))+2T_{-r}'(z) T_{\mu}'(\a T_{-r}(z))\a T_{-r}'(z))\\
&\qquad\qquad +T_{-r}(z) T_{\mu}''(\a T_{-r}(z))(\a T_{-r}'(z))^2+T_{-r}(z) T_{\mu}'(\a T_{-r}(z))\a T_{-r}''(z)\big)+\\
&\qquad T_{-r}(z) T_{\lambda}'''(T_{-r}(z) T_{\mu}(\a T_{-r}(z)))(T_{-r}'(z) T_{\mu}(\a T_{-r}(z))+T_{-r}(z) T_{\mu}'(\a T_{-r}(z))\a T_{-r}'(z)))^3\\
&\qquad 3T_{-r}(z) T_{\lambda}''(T_{-r}(z) T_{\mu}(\a T_{-r}(z)))(T_{-r}'(z) T_{\mu}(\a T_{-r}(z))+T_{-r}(z) T_{\mu}'(\a T_{-r}(z))\a T_{-r}'(z))\\
&\qquad\qquad\big(T_{-r}''(z) T_{\mu}(\a T_{-r}(z)))+2T_{-r}'(z) T_{\mu}'(\a T_{-r}(z))\a T_{-r}'(z))\\
&\qquad\qquad +T_{-r}(z) T_{\mu}''(\a T_{-r}(z))(\a T_{-r}'(z))^2+T_{-r}(z) T_{\mu}'(\a T_{-r}(z))\a T_{-r}''(z)\big)\\
&\qquad T_{-r}(z) T_{\lambda}'(T_{-r}(z) T_{\mu}(\a T_{-r}(z)))\cdot\\
&\qquad \big(T_{-r}'''(z) T_{\mu}(\a T_{-r}(z))+3T_{-r}''(z) T_{\mu}'(\a T_{-r}(z))\a T_{-r}'(z)
 +3T_{-r}'(z) T_{\mu}''(\a T_{-r}(z))(\a T_{-r}'(z))^2\\
&\qquad\qquad+3T_{-r}'(z) T_{\mu}'(\a T_{-r}(z))\a T_{-r}''(z)+T_{-r}(z) T_{\mu}'''(\a T_{-r}(z))(\a T_{-r}'(z))^3\\
&\qquad\qquad\qquad\qquad+3T_{-r}(z) T_{\mu}''(\a T_{-r}(z))\a^2 T_{-r}'(z)T_{-r}''(z)+T_{-r}(z) T_{\mu}'(\a T_{-r}(z))\a T_{-r}''''(z)\big).
\end{aligned}
\end{equation}
 and then
substituting $z=r$ into \eqref{appendix:g'''}, we have
\begin{align*}
&(T_{-u_0})'''(g(r)) (g'(r))^3+3(T_{-u_0})''(g(r)) g''(r)+T_{-u_0}'''(g(r)) g'''(r)\\
&\quad=T_{-r}'''(r) T_{\lambda}(0)+3T_{-r}''(r) T_{\lambda}'(0)(T_{-r}'(r) T_{\mu}(0))\\
&\quad\quad+3T_{-r}'(r) T_{\lambda}''(0)(T_{-r}'(r) T_{\mu}(0))^2+3T_{-r}'(r) T_{\lambda}'(0)\big(T_{-r}''(r) T_{\mu}(0)+2(T_{-r}'(r))^2 T_{\mu}'(0)\a \big).
\end{align*}
Together with $f'''(z)=3g''(z)+z g'''(z)$, we obtain
$$
f'''(r)= c_0+\rho_0\alpha.$$
Now $\alpha \in \ov{\D}$ is arbitrary, so the closed disk $\overline{\D}(c_0,\rho_0)$ is covered.

We know that $f'''(r)\in \partial \D(c_0, \rho_0)$ if and only if $f(z)=zg(z)$, where $g$ is a Blaschke product of degree 3 satisfying \eqref{condition2}, and then
we apply this fact to determine the precise form of $g$. Set
$$h(z)=\dfrac{T_{-u_0}\circ g \circ T_{r}(z)}{z},\quad z\in \D.$$
Clearly, $h$ is a Blaschke product of degree 2 depending on $g$ and satisfying
$$h(0)=(T_{-u_0}\circ g \circ T_{z_0})'(0)=v_0=\lambda.$$
Then $H(z)=T_{-v_0}\circ h(z)$ is a Blaschke product of degree 2 fixing $0$.  Set
$$G(z)=\frac{H(z)}{z}.$$
Obviously, $G$ is an automorphism of $\D$ depending on $g$ and satisfying
$$G(0)=H'(0)=T_{-v_0}'(v_0) h'(0)=\mu.$$
Thus $T_{-\mu}\circ G$ is an automorphism of $\D$ fixing 0,
which means that $T_{-\mu}\circ G(z)=e^{i \theta} z$ for $z\in \D$ and $\theta\in \R$.
Now it is easy to check that
$$g(z)=T_{u_0}\l(T_{-r}(z) T_{\lambda}(T_{-r}(z) T_{\mu}(\e T_{-r}(z)))\r), \quad z\in \D.$$
Conversely, if
$f(z)=z T_{u_0}\l(T_{-r}(z) T_{\lambda}(T_{-r}(z) T_{\mu}(\e T_{-r}(z)))\r)$, where $\th \in \R$, then direct calculations gives
$$f'''(r)=c_0+\rho_0 e^{i \theta}\in \partial \D(c_0, \rho_0).$$
Hence we complete the proof of this theorem.
\end{proof}
In fact, the immediate consequence of the argument above is the following corollary.
\begin{cor}\label{cor:third}
Let $0\le s<r<1$, $\lambda, \mu\in \ov{\D}$ with $w_1=\dfrac{s}{r}+\dfrac{r^2-s^2}{r(1-r^2)}\lambda$, $$w_2=\dfrac{2(r^2-s^2)}{r^2(1-r^2)^2}(\lambda(1-s\lambda)+r(1-|\lambda|^2)\mu).$$
 Suppose that $f\in\mathcal{H}_0$, $f(r)= s$, $f'(r)=w_1$ and $f''(r)=w_2$. Set $u_0=s/r$ and $v_0=\lambda$.
\begin{enumerate}
\item If $|\lambda|=1$, then $f'''(r)=c_0$ and $f(z)=z T_{u_0}(\lambda T_{-r}(z))$.

\item If $|\lambda|<1$, $|\mu|=1$, then $f'''(r)=c_0$ and $f(z)=z T_{u_0}\l(T_{-r}(z) T_{\lambda}(\mu T_{-r}(z))\r)$.

\item If $|\lambda|<1$, $|\mu|<1$, then the region of values of $f'''(z_0)$ is the closed disk
$\overline{\D}(c_0, \rho_0)$.
Furthermore, $f'''(z_0)\in \p\D(c_0, \rho_0)$ if and only if
$f(z)=z T_{u_0}\l(T_{-r}(z) T_{\lambda}(T_{-r}(z) T_{\mu}(\e T_{-r}(z)))\r)$, where $\theta \in \R$.
\end{enumerate}
\end{cor}
In addition, we obtain a sharp upper bound of $|f'''(r)|$ for Case (1).
\begin{rmk}
For $|\lambda|=1$,
\begin{align*}
|f'''(r)|
    &=\frac{6(r^2-s^2)}{r^3(1-r^2)^3}|(1+r^2)s\lambda^2-s^2\lambda^3-r^2\lambda|\\
    &\le \frac{6(r^2-s^2)}{r^3(1-r^2)^3}[(1+r^2)s+s^2+r^2],
\end{align*}
and equality holds if and only if $\lambda=-1$, or if and only if
$$
f(z)=-\frac{z-a}{1-az},$$
where $a=\dfrac{r^2+s}{r(1+s)}$.
\end{rmk}
We end this section by asking the meaningful question: is it possible to obtain a sharp upper bound for $|f'''(z)|$ depending only on $z$?

\section{Envelope of a family of circles}
Let $\beta\in \D$, we begin with analyzing the structure of the variability region
$$V(z_0,w_0,\beta)=\{f'''(z_0):f\in \H_0(z_0,w_0,\beta)\},$$
where
$$\H_0 (z_0,w_0,\beta) =  \{ f \in \H_0 : f(z_0) =w_0, f'(z_0)=\frac{w_0}{z_0}+\frac{r^2-s^2}{z_0(1-r^2)}\b\}.$$
We observe that the relation
$V(r,s,\lambda )=e^{i(3\vp-\xi)}V(z_0,w_0,\beta)$ holds, where $z_0=re^{i \vp}, w_0=s e^{i \xi}\in \D$ with $s<r$, then
it is sufficient to determine the variability region $V(r,s,\lambda )$, $\lambda\in \D$. Next we present some basic properties of it.

Since the class $\mathcal{H}_0(r,s,\lambda)$ is a compact convex subset of the linear space $\mathcal{A}$ of all analytic functions
in $\mathbb{D}$ endowed with the topology of
uniform convergence on compact subsets of $\mathbb{D}$ and $V(r,s,\lambda)$ is the image of $\mathcal{H}_0(r,s,\lambda)$ with respect to the continuous linear functional $\ell : \mathcal{A} \ni f \mapsto f'''(r) \in \mathbb{C}$,
$V(r,s,\lambda) = \ell ( \mathcal{H}_0 (r,s,\lambda))$ is also
a compact convex subset of $\mathbb{C}$. Furthermore, we claim that $V(r,s,\lambda )$ has nonempty interior, because
$\ov{\D}(AB,Ar^2(1-|\lambda|^2))\subset V(r,s,\lambda)$, where $A,B$
 are defined in \eqref{eq:AB}. Thus $V(r,s,\lambda )$ is a convex closed domain enclosed by the Jordan curve $\partial V(r,s,\lambda )$.

We define $c(\zeta),\rho(\zeta)$ and $V$ by
\begin{equation}\label{eq:c(zeta)-rho(zeta)}
c(\zeta)=\zeta (1-\eta\zeta),\quad \rho(\zeta)=t(1-|\zeta|^2),\quad V
 = \bigcup_{\zeta \in \overline{\mathbb{D}}}
  \overline{\mathbb{D}}(c( \zeta), \rho(\zeta )),
\end{equation}
where
$$\eta=\frac{r\ov{\lambda}}{1+r^2-2s\lambda},\quad  t=\frac{r}{|1+r^2-2s\lambda|}.$$
We remark that $\eta \in \C$, which is different from the case in \cite{chen2020}.
Then by the third order Dieudonn\'e Lemma, we have
$$V(r,s,\lambda ) = A
  \left(B+C V\right),$$
  where
\begin{equation}\label{eq:C}
C=r(1-|\lambda|^2)(1+r^2-2s\lambda)\in \C.
\end{equation}

 We claim that the set $V$ has the same properties as $V(r,s,\lambda )$.
Firstly, it is not difficult to see that
$V$ contains $\overline{\mathbb{D}}(0,t)$.
Secondly, the compactness and convexity of $V$ follows from
the fact $V$ corresponds to  the variability region
$V(r,s,\lambda )$.
Therefore we reduce the determination of $\partial V(r,s,\lambda)$ to
that of $V$.

Using the same method in \cite{chen2020}, we obtain the result below, analogous to \cite[Proposition 2.1 and 2.3]{chen2020}, which gives the parameter representation of $\p V$.
\begin{prop}
\label{prop:boundary-V}
For $\theta \in \mathbb{R}$, let
$t_\theta$ be the unique solution to the
equation
\begin{equation}
\label{eq:definig_equation_of_r_theta}
   |xe^{i\theta} - \ov{\eta}| = 2(x^2-|\eta|^2) , \quad x >|\eta| ,
\end{equation}
if $|xe^{i\theta} -\ov{\eta}| \geq  2(x^2-|\eta|^2) $;
otherwise let $t_\theta =t$.
Set
\begin{equation}
\label{eq:def_of_zeta_theta}
  \zeta_\theta = \frac{t_\theta e^{i\theta} -\ov{\eta} }{2(t_\theta^2-|\eta|^2)}
   \in \overline{\mathbb{D}} .
\end{equation}
Then $V$ is a convex closed domain enclosed by the Jordan curve $\partial V$.
Furthermore, $v_\theta\in \p V$ can be expressed as
\begin{equation}
\label{eq:expression_of_v_theta}
   v_\theta
    =
    \begin{cases}
     c( \zeta_\theta )  + \rho( \zeta_\theta ) e^{i \theta },  \quad &
     |te^{i \theta}-\ov{\eta}| < 2(t^2-|\eta|^2) , \\
    c( \zeta_\theta ) ,  \quad &
     |te^{i \theta}-\ov{\eta}| \geq 2(t^2-|\eta|^2) ,
    \end{cases}
\end{equation}
and the mapping
\[
   (-\pi , \pi ] \ni \th \mapsto v_\theta \in \partial V
\]
is a continuous bijection and gives a parametric representation of
$\partial V$.
\end{prop}
\section{Variability region for the third derivative}
With the auxiliary results in Sect. 3, we are ready to give the following theorem which gives the unified parametric representation of $\p V(r,s,\lambda)$ and the all extremal functions. It is worth mentioning that this result directly follows from Proposition \ref{prop:boundary-V} and is an analogue of \cite[Theorem 1.2]{chen2020}. Recall that $A,B,C$ are given in \eqref{eq:AB}, \eqref{eq:C} and $\zeta_\theta$ is defined in Proposition \ref{prop:boundary-V}.
\begin{thm}
\label{thm:boundary-curve}
Let $0 \leq s < r< 1$ and $|\lambda|<1$.
 Then $V(r,s,\lambda)$ is a convex closed domain enclosed by the Jordan curve $\partial V(r,s,\lambda)$ and a parametric representation
$(- \pi , \pi ] \ni \theta \mapsto \gamma ( \theta )$
of $\partial V(r,s,\lambda)$
 is given as follows.
\begin{enumerate}
 \item[{\rm (i)}]
If $t+|\eta|\leq \frac{1}{2} $, then $|te^{i \theta}-\ov{\eta}| \geq 2(t^2-|\eta|^2)$
for all $\theta \in \mathbb{R}$ and
$$
   \gamma( \theta )=
 A \left(B+C c( \zeta_\theta ) \right)   \in \partial V(r,s,\lambda).$$
  \item[{\rm (ii)}]
If $t-|\eta|\leq \frac{1}{2} $, then $|te^{i \theta}-\ov{\eta}| \le 2(t^2-|\eta|^2)$
for all $\theta \in \mathbb{R}$ and
 $$\gamma( \theta ) =
    A \left(B+C( c( \zeta_\theta )  + \rho(\zeta_\theta) e^{i \theta}) \right)   \in \partial V(r,s,\lambda).$$
 \item[{\rm (iii)}]
If  $t+|\eta| > \frac{1}{2}$ and $t-|\eta| < \frac{1}{2} $, then
\[
   \gamma( \theta ) =
  \begin{cases}
  A \left(B+C( c( \zeta_\theta )  + \rho(\zeta_\theta) e^{i \theta}) \right) , \quad
   & |te^{i \theta}-\ov{\eta}| < 2(t^2-|\eta|^2) , \\[2ex]
      A \left(B+C c( \zeta_\theta ) \right)  , \quad
   & |te^{i \theta}-\ov{\eta}|\ge 2(t^2-|\eta|^2).
  \end{cases}	
\]
\end{enumerate}
Furthermore,
\[
    f'''(r) =  A \left(B+C( c( \zeta_\theta )  + \rho(\zeta_\theta) e^{i \theta} )\right) \in \p V(r,s,\lambda),
\]
for some $\theta \in \mathbb{R}$ with $\zeta_\theta \in \mathbb{D}$
if and only if
$$
  f(z)=z T_{\frac{s}{r}}\l(T_{-r}(z) T_{\lambda}(T_{-r}(z) T_{\zth}(e^{i (\theta+\arg C)} T_{-r}(z)))\r), \quad z \in \mathbb{D}.
$$
Similarly
\[
    f'''(r) =   A \left(B+C c (\zeta_\theta) \right)\in \p V(r,s,\lambda),
\]
for some $\theta \in \mathbb{R}$ with $\zeta_\theta \in \partial \mathbb{D}$
if and only if
$$
  f(z)=z T_{\frac{s}{r}}\left(T_{-r}(z) T_{\lambda}(\zth T_{-r}(z))\right), \quad z \in \mathbb{D} .
$$
\end{thm}
The above theorem has a direct consequence corresponding to \cite[Theorem 1.1]{chen2020}, which shows three cases of $\p V(r,s,\lambda)$.
\begin{thm}
\label{thm:expression-of-boundary-curve}
Let $0 \leq s < r< 1$ and $|\lambda|<1$. Then
\begin{itemize}
 \item[{\rm (i)}]
If $t+|\eta|\leq \frac{1}{2} $,
then $\p V(r,s,\lambda)$ coincides with the Jordan curve given by
\begin{equation}\label{eq:i-p-V}
   \partial \mathbb{D} \ni \zeta
    \mapsto A \left(B+C c(\zeta )\right).
\end{equation}
 \item[{\rm (ii)}]
If $t-|\eta|\geq \frac{1}{2} $,
then $\partial V(r,s,\lambda)$ coincides with  the circle given by
\begin{equation}\label{eq:ii-p-V}
   \partial \mathbb{D} \ni \zeta
    \mapsto A
  \left\{B+\frac{\left[ 1+4(t^2-|\eta|^2) \right]t e^{i\th} -\ov{\eta}}{4(t^2-|\eta|^2)}C\right\}.
\end{equation}
 \item[{\rm (iii)}]
If $t+|\eta| > \frac{1}{2}$ and $t-|\eta| < \frac{1}{2} $,
then
$\p V(r,s,\lambda)$ consists of the circular arc given by
\begin{equation}\label{eq:iii-1-p-V}
   \Theta\ni \th
    \mapsto A
  \left\{B+\frac{\left[ 1+4(t^2-|\eta|^2) \right]t e^{i\th} -\ov{\eta}}{4(t^2-|\eta|^2)}C\right\},
\end{equation}
and the simple arc given by
\begin{equation}\label{eq:iii-2-p-V}
    J \ni \zeta
    \mapsto A \left(B+C c(\zeta )\right),
\end{equation}
where
$$\Theta=\left\{\theta\in (-\pi,\pi]: \cos \left(\theta+\arg(\eta)\right)> \frac{t^2+|\eta|^2-4(t^2-|\eta|^2)^2}{2t|\eta|} \right\},$$
and
$$J=\left\{\zth:\cos \left(\theta+\arg(\eta)\right)\leq \frac{t^2+|\eta|^2-4(t^2-|\eta|^2)^2}{2t|\eta|}\right\}$$
 is the closed subarc of $\partial \mathbb{D}$.
\end{itemize}
\end{thm}
\begin{rmk}
$J$ is the closed subarc of $\partial \mathbb{D}$ which has end points
$\zeta_{\theta_1} = \dfrac{te^{i \theta_1} -\ov{\eta} }{2(t^2-|\eta|^2)}$
and $\zeta_{\theta_2} = \dfrac{te^{i \theta_2} -\ov{\eta} }{2(t^2-|\eta|^2)}$,
where $-\pi<\th_1<\th_2\le \pi$ are the two solutions of
$$|te^{i\theta} - \ov{\eta}| = 2(t^2-|\eta|^2).$$
\end{rmk}

\subsection*{Acknowledgements}
The author would like to thank Professor Toshiyuki Sugawa for his helpful comments and valuable suggestions. The author is also grateful to Professor Ming Li for her helpful suggestions.
\bibliographystyle{amsplain} 

\end{document}